\documentclass[12pt]{amsart}

\usepackage{mathrsfs,amssymb}

\newtheorem{theorem}{Theorem}[section]
\newtheorem{lemma}[theorem]{Lemma}
\newtheorem{prop}[theorem]{Proposition}
\newtheorem{cor}[theorem]{Corollary}

\theoremstyle{definition}

\theoremstyle{remark}
\newtheorem{remark}[theorem]{Remark}

\numberwithin{equation}{section}

\DeclareMathOperator*{\esssup}{ess\,sup}

\let \la=\lambda
\let \e=\varepsilon
\let \d=\delta
\let \o=\omega
\let \a=\alpha
\let \f=\varphi

\let \O=\Omega
\let \si=\sigma

\let \ga=\gamma

\begin{document}
\title[On pointwise and weighted estimates]
{On pointwise and weighted estimates for commutators of Calder\'on-Zygmund operators}

\author{Andrei K. Lerner}
\address{Department of Mathematics,
Bar-Ilan University, 5290002 Ramat Gan, Israel}
\email{lernera@math.biu.ac.il}

\thanks{The first author was supported by the Israel Science Foundation (grant No. 953/13). The third author was supported by Grant MTM2012-30748, Spanish Government}

\author{Sheldy Ombrosi}
\address{Departamento de Matem\'atica\\
Universidad Nacional del Sur\\
Bah\'ia Blanca, 8000, Argentina}\email{sombrosi@uns.edu.ar}

\author{Israel P. Rivera-R\'{\i}os}
\address{IMUS \& Departamento de An\'alisis Matem\'atico, Universidad de Sevilla, Sevilla, Spain}
\email{petnapet@gmail.com}

\begin{abstract}
In recent years, it has been well understood that a Calder\'on-Zygmund operator $T$ is pointwise controlled
by a finite number of dyadic operators of a very simple structure (called the sparse operators).
We obtain a similar pointwise estimate for the commutator $[b,T]$ with a locally integrable function $b$.
This result is applied into two directions. If $b\in BMO$, we improve several weighted weak type bounds
for $[b,T]$. If $b$ belongs to the weighted $BMO$, we obtain a quantitative form of the two-weighted bound
for $[b,T]$ due to Bloom-Holmes-Lacey-Wick.
\end{abstract}



\maketitle

\section{Introduction}
\subsection{A pointwise bound for commutators}
In the past decade, a question about sharp weighted inequalities has leaded to a much better understanding of classical Calder\'on-Zygmund operators.
In particular, it was recently discovered by several authors (see  \cite{CR,HRT,La,L,LN}, and also \cite{BFP,CPO} for some interesting developments) that
a Calder\'on-Zygmund operator is dominated pointwise by a finite number of sparse operators ${\mathcal A}_{\mathcal S}$ defined by
$${\mathcal A}_{\mathcal S}f(x)=\sum_{Q\in {\mathcal S}}f_Q\chi_Q(x),$$
where $f_Q=\frac{1}{|Q|}\int_Qf$ and ${\mathcal S}$ is a sparse family of cubes from ${\mathbb R}^n$ (the latter means
that each cube $Q\in {\mathcal S}$ contains a set $E_Q$ of comparable measure and the sets $\{E_Q\}_{Q\in {\mathcal S}}$ are pairwise disjoint).

In this paper we obtain a similar domination result for the commutator $[b,T]$ of a Calder\'on-Zygmund operator $T$ with a locally integrable function $b$,
defined by
$$[b,T]f(x)=bTf(x)-T(bf)(x).$$
Then we apply this result in order to derive several new weighted weak and strong type inequalities for $[b,T]$.

Throughout the paper, we shall deal with $\omega$-Calder\'on-Zygmund operators $T$ on ${\mathbb R}^n$.
By this we mean that $T$ is $L^2$ bounded, represented~as
$$Tf(x)=\int_{{\mathbb R}^n}K(x,y)f(y)dy\quad\text{for all}\,\,x\not\in \text{supp}\,f,$$
with kernel $K$ satisfying the size condition
$|K(x,y)|\le \frac{C_K}{|x-y|^n},x\not=y,$ and the smoothness condition
$$|K(x,y)-K(x',y)|+|K(y,x)-K(y,x')|\le \o\left(\frac{|x-x'|}{|x-y|}\right)\frac{1}{|x-y|^n}$$
for $|x-y|>2|x-x'|$, where $\omega:[0,1]\to [0,\infty)$ is continuous, increasing, subadditive and $\o(0)=0$.

In \cite{La}, M. Lacey established a pointwise bound by sparse operators for $\o$-Calder\'on-Zygmund operators with $\o$ satisfying
the Dini condition $[\omega]_{\text{\rm{Dini}}}=\int_0^1\omega(t)\frac{dt}{t}<\infty.$ For such operators we adopt the notation
$$C_T=\|T\|_{L^2\to L^2}+C_K+[\omega]_{\text{\rm{Dini}}}.$$

A quantitative version of Lacey's result due to T. Hyt\"onen, L. Roncal and O. Tapiola \cite{HRT} states that
\begin{equation}\label{quant}
|Tf(x)|\le c_nC_T\sum_{j=1}^{3^n}{\mathcal A}_{\mathcal S_j}|f|(x).
\end{equation}
An alternative proof of this result was obtained by the first author \cite{L}.

In order to state an analogue of (\ref{quant}) for commutators, we introduce the sparse operator ${\mathcal T}_{\mathcal S,b}$
defined by
$${\mathcal T}_{{\mathcal S},b}f(x)=\sum_{Q\in {\mathcal S}}|b(x)-b_Q|f_Q\chi_Q(x).$$
Let ${\mathcal T}^{\star}_{{\mathcal S},b}$ denote the adjoint operator to ${\mathcal T}_{\mathcal S,b}$:
$${\mathcal T}^{\star}_{\mathcal S,b}f(x)=\sum_{Q\in {\mathcal S}}\left(\frac{1}{|Q|}\int_Q|b-b_Q|f\right)\chi_Q(x).$$

Our first main result is the following. Its proof is based on ideas developed in \cite{L}.

\begin{theorem}\label{commpoint}
Let $T$ be an $\omega$-Calder\'on-Zygmund operator with $\omega$ satisfying the Dini condition, and let $b\in L^1_{\text{loc}}$.
For every compactly supported $f\in L^{\infty}({\mathbb R}^n)$, there exist $3^n$ dyadic lattices ${\mathscr D}^{(j)}$
and $\frac{1}{2\cdot 9^n}$-sparse families ${\mathcal S}_j\subset {\mathscr D}^{(j)}$ such that for a.e. $x\in {\mathbb R}^n$,
\begin{equation}\label{pointb}
|[b,T]f(x)|\le c_nC_T\sum_{j=1}^{3^n}\big({\mathcal T}_{{\mathcal S}_j,b}|f|(x)+{\mathcal T}^{\star}_{{\mathcal S}_j,b}|f|(x)\big).
\end{equation}
\end{theorem}

Some comments about this result are in order. A classical theorem of R. Coifman, R. Rochberg and G.~Weiss~\cite{CRW}
says that the condition $b\in BMO$ is sufficient (and for some $T$ is also necessary) for the $L^p$ boundedness of $[b,T]$ for all $1<p<\infty$.
It is easy to see that the definition of ${\mathcal T}_{\mathcal S,b}$ is adapted to this condition. In Lemma \ref{tbf} below we show that if $b\in BMO$,
then ${\mathcal T}_{\mathcal S,b}$ is of weak type $(1,1)$. On the other hand, C. P\'erez \cite{P2} showed that $[b,T]$ is not of weak type $(1,1)$.
Therefore, the second term ${\mathcal T}^{\star}_{{\mathcal S}_j,b}$ cannot be removed from (\ref{pointb}).

Notice that the first term ${\mathcal T}_{\mathcal S,b}$ cannot be removed from (\ref{pointb}), too. Indeed, a standard argument (see the proof of (\ref{bmoest}) in
Section 2.2) based on the John-Nirenberg inequality shows that
if $b\in BMO$, then
$${\mathcal T}^{\star}_{\mathcal S,b}f(x)\le c_n\|b\|_{BMO}\sum_{Q\in {\mathcal S}}\|f\|_{L\log L,Q}\chi_Q(x).$$
But it was recently observed \cite{PRR2} that $[b,T]$ cannot be pointwise bounded by an $L\log L$-sparse operator appearing here.

In the following subsections we will show applications of Theorem~\ref{commpoint} to weighted weak and strong type inequalities for $[b,T]$.

\subsection{Improved weighted weak type bounds}
Given a weight $w$ (that is, a non-negative locally integrable function)
and a measurable set $E\subset {\mathbb R}^n$, denote $w(E)=\int_Ewdx$ and
$$w_f(\la)=w\{x\in {\mathbb R}^n:|f(x)|>\la\}.$$

In the classical work \cite{FS}, C. Fefferman and E.M. Stein obtained the following  weighted weak type $(1,1)$ property of
the Hardy-Littlewood maximal operator $M$: for an arbitrary weight $w$,
\begin{equation}\label{fs}
w_{Mf}(\la)\le \frac{c_n}{\la}\int_{{\mathbb R}^n}|f(x)|Mw(x)dx\quad(\la>0).
\end{equation}

Only forty years after that, M.C. Reguera and C. Thiele \cite{RT} gave an example showing that a similar estimate is not true
for the Hilbert transform instead of $M$ on the left-hand side of (\ref{fs}) (they disproved by this the so-called Muckenhoupt-Wheeden
conjecture). On the other hand, it was shown earlier by C. P\'erez \cite{P1}
that an analogue of (\ref{fs}) holds for a general class of Calder\'on-Zygmund operators but with a slightly bigger Orlicz maximal operator
$M_{L(\log L)^{\e}}$ instead of $M$ on the right-hand side. This result was reproved with a better dependence on~$\e$ in \cite{HP}:
if $T$ is a Calder\'on-Zygmund operator and $0<\e\le 1$, then
\begin{equation}\label{eq:HP}
w_{Tf}(\la)\le \frac{c(n,T)}{\e}\frac{1}{\la}\int_{\mathbb{R}^{n}}|f(x)|M_{L(\log L)^{\varepsilon}}w(x)dx\quad(\la>0).
\end{equation}

A general Orlicz maximal operator $M_{\f(L)}$ is defined for a Young function $\f$ by
$$M_{\f(L)}f(x)=\sup_{Q\ni x}\|f\|_{\f,Q},$$
where the supremum is taken over all cubes $Q\subset {\mathbb R}^n$ containing $x$, and
$\|f\|_{\f,Q}$ is the normalized Luxemburg norm defined by
$$
\|f\|_{\f,Q}=\inf\Big\{\la>0:\frac{1}{|Q|}\int_Q\f(|f(y)|/\la)dy\le 1\Big\}.
$$
If $\f(t)=t\log^{\a}({\rm{e}}+t), \a>0$, denote $M_{\f(L)}=M_{L(\log L)^{\a}}$.

Recently, C. Domingo-Salazar, M. Lacey and G. Rey \cite{DLR} obtained the following improvement of (\ref{eq:HP}):
if $C_{\f}=\int_{1}^{\infty}\frac{\f^{-1}(t)}{t^2\log(\rm{e}+t)}dt<\infty,$ then
\begin{equation}\label{border}
w_{Tf}(\la)\le \frac{c(n,T)C_{\f}}{\la}\int_{{\mathbb R}^n}|f(x)|M_{\f(L)}w(x)dx.
\end{equation}
It is easy to see that if $\f(t)=t\log^{\e}({\rm{e}}+t)$, then $C_{\f}\sim\frac{1}{\e}$, and thus (\ref{border})
contains (\ref{eq:HP}) as a particular case. On the other hand, (\ref{border}) holds for smaller functions
than $t\log^{\e}({\rm{e}}+t)$, for instance, for $\f(t)=t\log\log^{\a}({\rm e}^{\rm e}+t),\a>1$. The key ingredient in
the proof of (\ref{border}) was a pointwise control of $T$ by sparse operators expressed in (\ref{quant}).

Consider now the commutator $[b,T]$ of $T$ with a $BMO$ function $b$.
The following analogue of (\ref{eq:HP}) was recently obtained by the third author and C. P\'erez \cite{PRR1}:
for all $0<\e\le 1$,
\begin{equation}\label{eq:PRR}
w_{[b,T]f}(\la)\le\frac{c(n,T)}{\varepsilon^{2}}\int_{\mathbb{R}^{n}}\Phi\left(\|b\|_{BMO}\frac{|f(x)|}{\lambda}\right)M_{L(\log L)^{1+\varepsilon}}w(x)dx,
\end{equation}
where $\Phi(t)=t\log({\rm{e}}+t)$. Observe that $\Phi$ here reflects an unweighted $L\log L$ weak type estimate for $[b,T]$ obtained by C. P\'erez \cite{P2}.
Notice also that (\ref{eq:PRR}) with worst dependence on $\e$ was proved earlier in~\cite{PP}.

Similarly to the above mentioned improved weak type bound for Calder\'on-Zygmund operators (\ref{border}), we apply Theorem \ref{commpoint}
to improve (\ref{eq:PRR}). Our next result shows that (\ref{eq:PRR}) holds with $1/\e$ instead
of $1/\e^2$ and that $M_{L(\log L)^{1+\varepsilon}}$ in (\ref{eq:PRR}) can be replaced by smaller Orlicz maximal operators.

\begin{theorem}\label{weakcomm} Let $T$ be an $\o$-Calder\'on-Zygmund operator with $\o$ satisfying the Dini condition, and let $b\in BMO$.
Let $\f$ be an arbitrary Young function such that $C_{\f}=\int_{1}^{\infty}\frac{\f^{-1}(t)}{t^2\log(\rm{e}+t)}dt<\infty.$
Then for every weight $w$ and for every compactly supported $f\in L^{\infty}$,
\begin{equation}\label{impr}
w_{[b,T]f}(\la)\le c_nC_TC_{\f}\int_{\mathbb{R}^{n}}\Phi\left(\|b\|_{BMO}\frac{|f(x)|}{\lambda}\right)M_{(\Phi\circ\f)(L)}w(x)dx,
\end{equation}
where $\Phi(t)=t\log({\rm{e}}+t)$.
\end{theorem}

By Theorem \ref{commpoint}, the proof of (\ref{impr}) is based on weak type estimates for ${\mathcal T}_{\mathcal S,b}$ and
${\mathcal T}^{\star}_{\mathcal S,b}$. The maximal operator $M_{(\Phi\circ\f)(L)}$ appears in the weighted weak type $(1,1)$ estimate for ${\mathcal T}_{\mathcal S,b}$.
It is interesting that ${\mathcal T}^{\star}_{\mathcal S,b}$, being not of weak type $(1,1)$, satisfies
a better estimate than (\ref{impr}) with a smaller maximal operator than $M_{(\Phi\circ\f)(L)}$ (which one can deduce from Lemma \ref{lloglsparse} below).

We mention several particular cases of interest in Theorem \ref{weakcomm}. Notice that if $\f(t)\le t^2$ for $t\ge t_0$,  then
$$\Phi\circ\f(t)\le c\f(t)\log({\rm e}+t)\quad(t>0).$$
Hence, if $\f(t)=t\log^{\e}({\rm e}+t)$, $0<\e\le 1$, then simple estimates along with (\ref{impr}) imply
\begin{equation}\label{caseof}
w_{[b,T]f}(\la)\le \frac{c(n,T)}{\e}\int_{\mathbb{R}^{n}}\Phi\left(\|b\|_{BMO}\frac{|f(x)|}{\lambda}\right)M_{L(\log L)^{1+\varepsilon}}w(x)dx.
\end{equation}
Similarly, if $\f(t)=t\log\log^{1+\e}({\rm e}^{\rm e}+t),0<\e\le 1$, then
$$
w_{[b,T]f}(\la)\le \frac{c(n,T)}{\e}\int_{\mathbb{R}^{n}}\Phi\left(\|b\|_{BMO}\frac{|f(x)|}{\lambda}\right)M_{L(\log L)(\log\log L)^{1+\varepsilon}}w(x)dx.
$$

As an application of Theorem \ref{weakcomm}, we obtain an improved weighted weak type estimate for $[b,T]$ assuming that the weight $w\in A_1$. Recall that the latter condition
means that
$$[w]_{A_1}=\sup_{x\in {\mathbb R}^n}\frac{Mw(x)}{w(x)}<\infty.$$
Also we define the $A_{\infty}$ constant of $w$ by
$$[w]_{A_{\infty}}=\sup_{Q}\frac{1}{w(Q)}\int_QM(w\chi_Q)dx,$$
where the supremum is taken over all cubes $Q\subset {\mathbb R}^n$.
It was shown in~\cite{HP} that the dependence on $\e$ in (\ref{eq:HP}) implies the corresponding mixed $A_1$-$A_{\infty}$ estimate. In a similar way we have the following.

\begin{cor}\label{imprdep}
For every $w\in A_1$,
$$
w_{[b,T]f}(\la)\le c_nC_T[w]_{A_1}\Phi([w]_{A_{\infty}})\int_{{\mathbb R}^n}
\Phi\left(\|b\|_{BMO}\frac{|f(x)|}{\lambda}\right)w(x)dx,
$$
where $\Phi(t)=t\log({\rm e}+t)$.
\end{cor}

This provides a logarithmic improvement of the corresponding bounds in \cite{OC,PRR1}.

\subsection{Two-weighted strong type bounds} Let $w$ be a weight, and let $1<p<\infty$. Denote $\si_w(x)=w^{-\frac{1}{p-1}}(x)$.
Given a cube $Q\subset {\mathbb R}^n$, set
$$[w]_{A_p,Q}=\frac{w(Q)}{|Q|}\left(\frac{\si_w(Q)}{|Q|}\right)^{p-1}.$$
We say that $w\in A_p$ if
$$[w]_{A_p}=\sup_{Q}[w]_{A_p,Q}<\infty.$$

As we have mentioned previously, pointwise bounds by sparse operators were motivated by sharp weighted norm inequalities. For example, (\ref{quant}) provides
a simple proof of the sharp $L^p(w)$ bound for $T$ (see \cite{HRT,L}):
\begin{equation}\label{hyt}
\|T\|_{L^p(w)}\le c_{n,p}C_T[w]_{A_p}^{\max\big(1,\frac{1}{p-1}\big)}.\quad(1<p<\infty)
\end{equation}
In the case of $\o$-Calder\'on-Zygmund operators with $\o(t)=ct^{\d}$, (\ref{hyt}) was proved by T. Hyt\"onen \cite{H} (see also \cite{Hyt2, Le} for
the history of this result and a different proof).

An analogue of (\ref{hyt}) for the commutator $[b,T]$ with a $BMO$ function~$b$ is the following sharp $L^p(w)$ bound due to D. Chung, C. Pereyra and C. P\'erez \cite{CPP}:
\begin{equation}\label{cpp}
\|[b,T]\|_{L^p(w)}\le c(n,p,T)\|b\|_{BMO}[w]_{A_p}^{2\max\big(1,\frac{1}{p-1}\big)}.\quad(1<p<\infty)
\end{equation}

Much earlier, S. Bloom \cite{SB} obtained an  interesting two-weighted result for the commutator of the Hilbert transform $H$:
if $\mu, \la\in A_p, 1<p<\infty, \nu=(\mu/\la)^{1/p}$ and $b\in BMO_{\nu}$, then
\begin{equation}\label{bloom}
\|[b,H]f\|_{L^p(\la)}\le c(p,\mu,\la)\|b\|_{BMO_{\nu}}\|f\|_{L^p(\mu)}.
\end{equation}
Here $BMO_{\nu}$ is the weighted $BMO$ space of locally integrable functions $b$ such that
$$\|b\|_{BMO_{\nu}}=\sup_{Q}\frac{1}{\nu(Q)}\int_Q|b-b_Q|dx<\infty.$$

Recently, I. Holmes, M. Lacey and B. Wick \cite{HLW} extended (\ref{bloom}) to $\o$-Calder\'on-Zygmund operators with $\o(t)=ct^{\d}$;
the key role in their proof was played by Hyt\"onen's representation theorem \cite{H} for such operators.
In the particular case when $\mu=\la=w\in A_2$ the approach in \cite{HLW} recovers (\ref{cpp}) (this was checked in \cite{HW};
and also, (\ref{bloom}) was extended in this work to higher-order commutators).

Using Theorem \ref{commpoint}, we obtain the following quantitative version of the Bloom-Holmes-Lacey-Wick result. It extends (\ref{bloom})
to any $\o$-Calder\'on-Zygmund operator with the Dini condition, and the explicit dependence on $[\mu]_{A_p}$ and $[\la]_{A_p}$ is found. Also, it can be viewed as a natural extension of (\ref{cpp})
to the two-weighted setting.

\begin{theorem}\label{bhlw} Let $T$ be an $\omega$-Calder\'on-Zygmund operator with $\o$ satisfying the Dini condition.
Let $\mu, \la\in A_p, 1<p<\infty,$ and $\nu=(\mu/\la)^{1/p}$. If $b\in BMO_{\nu}$, then
$$
\|[b,T]f\|_{L^p(\la)}\le c_{n,p}C_T\big([\mu]_{A_p}[\la]_{A_p}\big)^{\max\big(1,\frac{1}{p-1}\big)}\|b\|_{BMO_{\nu}}\|f\|_{L^p(\mu)}.
$$
\end{theorem}

\vskip 5mm
The paper is organized as follows. In Section 2, we collect some preliminary information about dyadic lattices, sparse families and Young functions.
Section 3 is devoted to the proof of Theorem \ref{commpoint}. In Section~4, we prove Theorem \ref{weakcomm} and Corollary \ref{imprdep}, and Section 5
contains the proof of Theorem \ref{bhlw}.

\section{Preliminaries}
\subsection{Dyadic lattices and sparse families}
By a cube in ${\mathbb R}^n$ we mean a half-open cube $Q=\prod_{i=1}^n[a_i,a_i+h), h>0$. Denote by $\ell_Q$ the sidelength of $Q$.
Given a cube $Q_0\subset {\mathbb R}^n$, let ${\mathcal D}(Q_0)$ denote the set of all dyadic cubes with respect to $Q_0$, that is, the cubes
obtained by repeated subdivision of $Q_0$ and each of its descendants into $2^n$ congruent subcubes.

A dyadic lattice ${\mathscr D}$ in ${\mathbb R}^n$ is any collection of cubes such that
\begin{enumerate}
\renewcommand{\labelenumi}{(\roman{enumi})}
\item
if $Q\in{\mathscr D}$, then each child of $Q$ is in ${\mathscr D}$ as well;
\item
every 2 cubes $Q',Q''\in {\mathscr D}$ have a common ancestor, i.e., there exists $Q\in{\mathscr D}$ such that $Q',Q''\in {\mathcal D}(Q)$;
\item
for every compact set $K\subset {\mathbb R}^n$, there exists a cube $Q\in {\mathscr D}$ containing $K$.
\end{enumerate}

For this definition, as well as for the next Theorem, we refer to \cite{LN}.

\begin{theorem}\label{three}{\rm{(The Three Lattice Theorem)}}
For every dyadic lattice ${\mathscr D}$, there exist $3^n$ dyadic lattices ${\mathscr D}^{(1)},\dots,{\mathscr D}^{(3^n)}$ such that
$$\{3Q: Q\in{\mathscr D}\}=\cup_{j=1}^{3^n}{\mathscr D}^{(j)}$$
and for every cube $Q\in {\mathscr D}$ and $j=1,\dots,3^n$, there exists a unique cube $R\in {\mathscr D}^{(j)}$ of
sidelength $\ell_{R}=3\ell_Q$ containing $Q$.
\end{theorem}

\begin{remark}\label{use}
Fix a dyadic lattice ${\mathscr D}$. For an arbitrary cube $Q\subset{\mathbb R}^n$, there is a cube $Q'\in {\mathscr D}$ such that
$\ell_Q/2<\ell_{Q'}\le \ell_Q$ and $Q\subset 3Q'$. By Theorem \ref{three}, there is $j=1,\dots,3^n$ such that $3Q'=P\in {\mathscr D}^{(j)}$.
Therefore, for every cube $Q\subset{\mathbb R}^n$, there exists $P\in {\mathscr D}^{(j)},j=1,\dots,3^n,$ such that $Q\subset P$ and $\ell_P\le 3\ell_Q$.
A similar statement can be found in \cite[Lemma 2.5]{HLP}.
\end{remark}

We say that a family ${\mathcal S}$ of cubes from ${\mathscr D}$ is $\eta$-sparse, $0<\eta<1$, if for every
$Q\in {\mathcal S}$, there exists a measurable set $E_Q\subset Q$ such that $|E_Q|\ge \eta|Q|$, and the sets
$\{E_Q\}_{Q\in {\mathcal S}}$ are pairwise disjoint.

A family ${\mathcal S}\subset {\mathscr D}$ is called $\Lambda$-Carleson, $\Lambda>1$, if for every cube $Q\in {\mathscr D}$,
$$\sum_{P\in {\mathcal S}, P\subset Q}|P|\le\Lambda|Q|.$$

It is easy to see that every $\eta$-sparse family is $(1/\eta)$-Carleson. In \cite[Lemma 6.3]{LN}, it is shown that the converse statement is
also true, namely, every $\Lambda$-Carleson family is $(1/\Lambda)$-sparse. Also, \cite[Lemma~6.6]{LN} says that if ${\mathcal S}$ is $\Lambda$-Carleson
and $m\in {\mathbb N}$ such that $m\ge 2$, then ${\mathcal S}$ can be written as a union of $m$ families ${\mathcal S}_j$, each of which is $(1+\frac{\Lambda-1}{m})$-Carleson.
Using the above mentioned relation between sparse and Carleson families, one can rewrite the latter fact as follows.

\begin{lemma}\label{unsparse}
If ${\mathcal S}\subset {\mathscr D}$ is $\eta$-sparse and $m\ge 2$, then one can represent ${\mathcal S}$ as a disjoint union ${\mathcal S}=\cup_{j=1}^m{\mathcal S}_j$, where
each family ${\mathcal S}_j$ is $\frac{m}{m+(1/\eta)-1}$-sparse.
\end{lemma}

Now we turn our attention to augmentation. Given a family of cubes~$\mathcal{S}$ contained in a dyadic lattice $\mathscr{D}$,  we associate to each cube $Q\in\mathcal{S}$ a family $\mathcal{F}(Q)\subseteq\mathcal{D}(Q)$ such that $Q\in\mathcal{F}(Q)$. In some situations it is useful to construct a new family that combines the families $\mathcal{F}(Q)$ and~$\mathcal{S}$. One way to build such a family is the following.

For each $\mathcal{F}(Q)$ let $\widetilde{\mathcal{F}}(Q)$ be the family that consists of all cubes $P\in\mathcal{F}(Q)$ that are not contained in any cube $R\in\mathcal{S}$ with $R\subsetneq Q$. Now we can define the augmented family $\widetilde{\mathcal{S}}$ as $$\widetilde{\mathcal{S}}=\bigcup_{Q\in\mathcal{S}}\widetilde{\mathcal{F}}(Q).$$ It is clear, by construction, that the augmented family  $\widetilde{\mathcal{S}}$ contains the original family $\mathcal{S}$. Furthermore, if  $\mathcal{S}$ and each $\mathcal{F}(Q)$ are sparse families, then the augmented family $\widetilde{\mathcal{S}}$ is also sparse. We state this fact more clearly in the following lemma
(see \cite[Lemma 6.7]{LN} and the above equivalence between the notions of the $\Lambda$-Carleson and $\frac{1}{\Lambda}$-sparse families).

\begin{lemma}\label{Lem:Aug}
If $\mathcal{S}\subset\mathscr{D}$ is an $\eta_0$-sparse family then  the augmented family $\widetilde{\mathcal{S}}$ built upon $\eta$-sparse families ${\mathcal{F}}(Q),Q\in {\mathcal S},$ is an $\frac{\eta\eta_0}{1+\eta_0}$-sparse family.
\end{lemma}

\subsection{Young functions and normalized Luxemburg norms}
By a Young function we mean a continuous,
convex, strictly increasing function $\f:[0,\infty)\to [0,\infty)$
with $\f(0)=0$ and $\f(t)/t\to \infty$ as $t\to \infty$.
Notice that such functions are also called in the literature the $N$-functions.
We refer to \cite{KR,RR} for their basic properties. We will use, in particular, that
$\f(t)/t$ is also a strictly increasing function (see, e.g., \cite[p. 8]{KR}).

We will also use the fact that
\begin{equation}\label{fact}
\|f\|_{\f,Q}\le 1\Leftrightarrow \frac{1}{|Q|}\int_Q\f(|f(x)|)dx\le 1.
\end{equation}

Given a Young function $\f$, its complementary function is defined by
$$\bar\f(t)=\sup_{x\ge 0}\big(xt-\f(x)\big).$$
Then $\bar\f$ is also a Young function satisfying $t\le \bar\f^{-1}(t)\f^{-1}(t)\le 2t$.
Also the following H\"older type estimate holds:
\begin{equation}\label{prop2}
\frac{1}{|Q|}\int_Q|f(x)g(x)|dx\le 2\|f\|_{\f,Q}\|g\|_{\bar\f,Q}.
\end{equation}

Recall that the John-Nirenberg inequality (see, e.g., \cite[p. 124]{G2}) says that for every $b\in BMO$
and for any cube $Q\subset {\mathbb R}^n$,
\begin{equation}\label{jn}
|\{x\in Q:|b(x)-b_Q|>\la\}|\le e|Q|e^{-\frac{\la}{2^ne\|b\|_{BMO}}}\quad(\la>0).
\end{equation}
In particular, this inequality implies (see \cite[p. 128]{G2})
$$
\frac{1}{|Q|}\int_Qe^{\frac{|b(x)-b_Q|}{c_n\|b\|_{BMO}}-1}dx\le 1.
$$
From this and from (\ref{fact}), taking $\f(t)=e^t-1$, we obtain
$$\|b-b_Q\|_{\f,Q}\le c_n\|b\|_{BMO}.$$
A simple computation shows that in this case $\bar\f(t)\approx t\log({\rm e}+t)$,
and therefore, by (\ref{prop2}),
\begin{equation}\label{bmoest}
\frac{1}{|Q|}\int_Q|(b-b_Q)g|dx\le c_n\|b\|_{BMO}\|g\|_{L\log L,Q}.
\end{equation}

Notice that many important properties of the Luxemburg normalized norms
$\|f\|_{\f,Q}$ hold without assuming the convexity of $\f$.
In particular, we will use the following generalized H\"older inequality.

\begin{lemma}\label{holder} Let $A,B$ and $C$ be non-negative, continuous, strictly increasing functions on $[0,\infty)$
satisfying $A^{-1}(t)B^{-1}(t)\le C^{-1}(t)$ for all $t\ge 0$. Assume also that $C$ is convex. Then
\begin{equation}\label{prop3}
\|fg\|_{C,Q}\le 2\|f\|_{A,Q}\|g\|_{B,Q}.
\end{equation}
\end{lemma}

This lemma was proved by R. O'Neil \cite{ON} under the assumption that $A,B$ and $C$ are Young functions but the same proof works under the above conditions.
Indeed, by homogeneity, it suffices to assume that $\|f\|_{A,Q}=\|g\|_{B,Q}=1$. Next, notice that the assumptions on $A,B$ and $C$ easily imply that
$C(xy)\le A(x)+B(y)$ for all $x,y\ge 0$. Therefore, using the convexity of $C$ and (\ref{fact}), we obtain
$$\frac{1}{|Q|}\int_QC(|fg|/2)dx\le \frac{1}{2}\Big(\frac{1}{|Q|}\int_QA(|f|)dx+\frac{1}{|Q|}\int_QB(|g|)dx\Big)\le 1,$$
which, by (\ref{fact}) again, implies (\ref{prop3}).

Given a dyadic lattice ${\mathscr D}$, denote
$$M_{\Phi}^{{\mathscr D}}f(x)=\sup_{Q\ni x, Q\in {\mathscr D}}\|f\|_{\Phi,Q}.$$
The following lemma is a generalization of the Fefferman-Stein inequality (\ref{fs}) to general Orlicz maximal functions,
and it is apparently well-known. We give its proof for the sake of completeness.

\begin{lemma} \label{LemmaAux} Let $\Phi$ be a Young function. For an arbitrary weight $w$,
$$
w\left\{ x\in\mathbb{R}^{n}\,:\,M_{\Phi}f(x)>\lambda\right\}\le 3^n\int_{{\mathbb R}^n}\Phi\left(\frac{9^n|f(x)|}{\lambda}\right)Mw(x)dx.
$$
\end{lemma}

\begin{proof}
By the Calder\'on-Zygmund decomposition adapted to $M_{\Phi}^{{\mathscr D}}$ (see \cite[p. 237]{CUMP}),
there exists a family of disjoint cubes $\{Q_{i}\}$ such that
$$
\left\{ x\in\mathbb{R}^{n}\,:\,M_{\Phi}^{\mathscr{D}}f(x)>\lambda\right\} =\cup_{i}Q_{i}
$$
and $\lambda<\|f\|_{\Phi,Q_{i}}\le 2^{n}\lambda.$
By (\ref{fact}), we see that $\|f\|_{\Phi,Q_{i}}>\la$ implies $\int_{Q_i}\Phi(|f|/\la)>|Q_i|$. Therefore,
\begin{eqnarray*}
&&w\{x\in\mathbb{R}^{n}:M_{\Phi}^{\mathscr{D}}f(x)>\lambda\}=\sum_iw(Q_i)\\
&&<\sum_iw_{Q_i}\int_{Q_i}\Phi(|f(x)|/\la)dx\le\int_{{\mathbb R}^n}\Phi(|f(x)|/\la)Mw(x)dx.
\end{eqnarray*}

Now we observe that by the convexity of $\Phi$ and Remark \ref{use}, there exist $3^n$ dyadic lattices ${\mathscr D}^{(j)}$ such that
$$
M_{\Phi}f(x)\le 3^n\sum_{j=1}^{3^{n}}M_{\Phi}^{{\mathscr D}^{(j)}}f(x).
$$
Combining this estimate with the previous one completes the proof.
\end{proof}

\begin{remark}\label{llog}
Suppose that $\Phi(t)=t\log({\rm e}+t)$. It is easy to see that for all $a,b\ge 0$,
\begin{equation}\label{subm}
\Phi(ab)\le 2\Phi(a)\Phi(b).
\end{equation}
From this and from Lemma \ref{LemmaAux},
$$
w\{x\in {\mathbb R}^n:M_{L\log L}f(x)>\la\}\le c_n\int_{{\mathbb R}^n}\Phi\left(\frac{|f(x)|}{\lambda}\right)Mw(x)dx.
$$
\end{remark}

\section{Proof of Theorem \ref{commpoint}}
The proof of Theorem~\ref{commpoint} is a slight modification of the argument in \cite{L}.
Although some parts of the proofs here and in \cite{L} are almost identical, certain details are different, and hence we give a complete proof.
We start by defining several important objects.

Let $T$ be an $\omega$-Calder\'on-Zygmund operator with $\omega$ satisfying the Dini condition.
Recall that the maximal truncated operator $T^{\star}$ is defined by
$$T^{\star}f(x)=\sup_{\e>0}\Big|\int_{|y-x|>\e}K(x,y)f(y)dy\Big|.$$

Define the grand maximal truncated operator ${\mathcal M}_T$ by
$${\mathcal M}_Tf(x)=\sup_{Q\ni x}\,\esssup_{\xi\in Q}|T(f\chi_{{\mathbb R}^n\setminus 3Q})(\xi)|,$$
where the supremum is taken over all cubes $Q\subset {\mathbb R}^n$ containing $x$.

Given a cube $Q_0$, for $x\in Q_0$ define a local version of ${\mathcal M}_T$ by
$${\mathcal M}_{T,Q_0}f(x)=\sup_{Q\ni x, Q\subset Q_0}\esssup_{\xi\in Q}|T(f\chi_{3Q_0\setminus 3Q})(\xi)|.$$

The next lemma was proved in \cite{L}.

\begin{lemma}\label{mainl}
The following pointwise estimates hold:
\begin{enumerate}
\renewcommand{\labelenumi}{(\roman{enumi})}
\item for a.e. $x\in Q_0$,
$$
|T(f\chi_{3Q_0})(x)|\le c_{n}\|T\|_{L^1\to L^{1,\infty}}|f(x)|+{\mathcal M}_{T,Q_0}f(x);
$$
\item for all $x\in {\mathbb R}^n$,
$$
{\mathcal M}_Tf(x)\le c_n(\|\omega\|_{\text{\rm{Dini}}}+C_K)Mf(x)+T^{\star}f(x).
$$
\end{enumerate}
\end{lemma}

An examination of standard proofs (see, e.g., \cite[Ch. 8.2]{G2}) shows that
\begin{equation}\label{weaktype}
\max(\|T\|_{L^1\to L^{1,\infty}},\|T^{\star}\|_{L^1\to L^{1,\infty}})\le c_nC_T.
\end{equation}
By part (ii) of Lemma \ref{mainl} and by (\ref{weaktype}),
\begin{equation}\label{weakmt}
\|{\mathcal M}_T\|_{L^1\to L^{1,\infty}}\le c_nC_T.
\end{equation}

\begin{proof}[Proof of Theorem \ref{commpoint}]
By Remark \ref{use}, there exist $3^n$ dyadic lattices ${\mathscr D}^{(j)}$ such that
for every $Q\subset {\mathbb R}^n$, there is a cube $R=R_Q\in {\mathscr D}^{(j)}$ for some $j$, for which
$3Q\subset R_Q$ and $|R_Q|\le 9^n|Q|$.

Fix a cube $Q_0\subset {\mathbb R}^n$. Let us show that there exists a
$\frac{1}{2}$-sparse family ${\mathcal F}\subset {\mathcal D}(Q_0)$ such that for a.e. $x\in Q_0$,
\begin{eqnarray}
&&|[b,T](f\chi_{3Q_0})(x)|\label{first}\\
&&\le c_nC_T\sum_{Q\in {\mathcal F}}\big(|b(x)-b_{R_Q}||f|_{3Q}+|(b-b_{R_Q})f|_{3Q}\big)\chi_Q(x)\nonumber.
\end{eqnarray}

It suffices to prove the following recursive claim:
there exist pairwise disjoint cubes $P_j\in {\mathcal D}(Q_0)$ such that
$\sum_j|P_j|\le\frac{1}{2}|Q_0|$ and
\begin{eqnarray*}
|[b,T](f\chi_{3Q_0})(x)|\chi_{Q_0}&\le& c_nC_T\big(|b(x)-b_{R_{Q_0}}||f|_{3Q_0}+|(b-b_{R_{Q_0}})f|_{3Q_0}\big)\\
&+&\sum_j|[b,T](f\chi_{3P_j})(x)|\chi_{P_j}.
\end{eqnarray*}
a.e. on $Q_0$. Indeed, iterating this estimate, we immediately get (\ref{first}) with ${\mathcal F}=\{P_j^k\},k\in {\mathbb Z}_+$, where
$\{P_j^0\}=\{Q_0\}$, $\{P_j^1\}=\{P_j\}$ and $\{P_j^k\}$ are the cubes obtained at the $k$-th stage of the iterative process.

Next, observe that for arbitrary pairwise disjoint cubes $P_j\in {\mathcal D}(Q_0)$,
\begin{eqnarray*}
&&|[b,T](f\chi_{3Q_0})|\chi_{Q_0}=|[b,T](f\chi_{3Q_0})|\chi_{Q_0\setminus \cup_jP_j}+\sum_j|[b,T](f\chi_{3Q_0})|\chi_{P_j}\\
&&\le |[b,T](f\chi_{3Q_0})|\chi_{Q_0\setminus \cup_jP_j}+\sum_j|[b,T](f\chi_{3Q_0\setminus 3P_j})|\chi_{P_j}\\
&&+\sum_j|[b,T](f\chi_{3P_j})|\chi_{P_j}.
\end{eqnarray*}
Hence, in order to prove the recursive claim,
it suffices to show that one can select pairwise disjoint cubes $P_j\in {\mathcal D}(Q_0)$ with
$\sum_j|P_j|\le\frac{1}{2}|Q_0|$ and such that for a.e. $x\in Q_0$,
\begin{eqnarray}
&&|[b,T](f\chi_{3Q_0})|\chi_{Q_0\setminus \cup_jP_j}+\sum_j|[b,T](f\chi_{3Q_0\setminus 3P_j})|\chi_{P_j}\label{suff}\\
&&\le c_nC_T\big(|b(x)-b_{R_{Q_0}}||f|_{3Q_0}+|(b-b_{R_{Q_0}})f|_{3Q_0}\big).\nonumber
\end{eqnarray}

Using that $[b,T]f=[b-c,T]f$ for any $c\in {\mathbb R}$, we obtain
\begin{eqnarray}
&&|[b,T](f\chi_{3Q_0})|\chi_{Q_0\setminus \cup_jP_j}+\sum_j|[b,T](f\chi_{3Q_0\setminus 3P_j})|\chi_{P_j}\label{a}\\
&&\le |b-b_{R_{Q_0}}|\Big(|T(f\chi_{3Q_0})|\chi_{Q_0\setminus\cup_jP_j}+\sum_j|T(f\chi_{3Q_0\setminus 3P_j})|\chi_{P_j}\Big)\nonumber\\
&&+|T((b-b_{R_{Q_0}})f\chi_{3Q_0})|\chi_{Q_0\setminus\cup_jP_j}+\sum_j|T((b-b_{R_{Q_0}})f\chi_{3Q_0\setminus 3P_j})|\chi_{P_j}.\nonumber
\end{eqnarray}

By (\ref{weakmt}), one can choose $\a_n$ such that the set $E=E_1\cup E_2$, where
$$E_1=\{x\in Q_0:|f|>\a_n|f|_{3Q_0}\}\cup \{x\in Q_0:{\mathcal M}_{T,Q_0}f>\a_nC_T|f|_{3Q_0}\}$$
and
\begin{eqnarray*}
E_2&=&\{x\in Q_0:|(b-b_{R_{Q_0}})f|>\a_n|(b-b_{R_{Q_0}})f|_{3Q_0}\}\\
&\cup& \{x\in Q_0:{\mathcal M}_{T,Q_0}(b-b_{R_{Q_0}})f>\a_nC_T|(b-b_{R_{Q_0}})f|_{3Q_0}\},
\end{eqnarray*}
will satisfy $|E|\le \frac{1}{2^{n+2}}|Q_0|$.

The Calder\'on-Zygmund decomposition applied to the function $\chi_E$ on $Q_0$ at height $\la=\frac{1}{2^{n+1}}$
produces pairwise disjoint cubes $P_j\in {\mathcal D}(Q_0)$ such that
$$\frac{1}{2^{n+1}}|P_j|\le |P_j\cap E|\le \frac{1}{2}|P_j|$$
and $|E\setminus \cup_jP_j|=0$. It follows that $\sum_j|P_j|\le \frac{1}{2}|Q_0|$ and $P_j\cap E^{c}\not=\emptyset$.
Therefore,
$$\esssup_{\xi\in P_j}|T(f\chi_{3Q_0\setminus 3P_j})(\xi)|\le c_nC_T|f|_{3Q_0}$$
and
$$\esssup_{\xi\in P_j}|T((b-b_{R_{Q_0}})f\chi_{3Q_0\setminus 3P_j})(\xi)|\le c_nC_T|(b-b_{R_{Q_0}})f|_{3Q_0}.$$

Also, by part (i) of Lemma \ref{mainl} and by (\ref{weaktype}), for a.e. $x\in Q_0\setminus \cup_jP_j$,
$$|T(f\chi_{3Q_0})(x)|\le c_nC_T|f|_{3Q_0}$$
and
$$|T((b-b_{R_{Q_0}})f\chi_{3Q_0})(x)|\le c_nC_T|(b-b_{R_{Q_0}})f|_{3Q_0}.$$

Combining the obtained estimates with (\ref{a}) proves (\ref{suff}), and therefore, (\ref{first}) is proved.

Take now a partition of ${\mathbb R}^n$ by cubes $Q_j$ such that $\text{supp}\,(f)\subset 3Q_j$ for each $j$. For example, take a cube $Q_0$ such that
$\text{supp}\,(f)\subset Q_0$ and cover $3Q_0\setminus Q_0$ by $3^n-1$ congruent cubes $Q_j$. Each of them satisfies $Q_0\subset 3Q_j$. Next, in the same way cover
$9Q_0\setminus 3Q_0$, and so on. The union of resulting cubes, including $Q_0$, will satisfy the desired property.

Having such a partition, apply (\ref{first}) to each $Q_j$. We obtain a $\frac{1}{2}$-sparse family ${\mathcal F}_j\subset {\mathcal D}(Q_j)$ such that
(\ref{first}) holds for a.e. $x\in Q_j$ with $|Tf|$ on the left-hand side. Therefore, setting ${\mathcal F}=\cup_{j}{\mathcal F}_j$, we obtain
that ${\mathcal F}$ is a $\frac{1}{2}$-sparse family, and for a.e. $x\in {\mathbb R}^n$,
\begin{equation}\label{glob}
|[b,T]f(x)|\le c_nC_T\sum_{Q\in {\mathcal F}}\big(|b(x)-b_{R_Q}||f|_{3Q}+|(b-b_{R_Q})f|_{3Q}\big)\chi_Q(x).
\end{equation}

Since $3Q\subset R_Q$ and $|R_Q|\le 3^n|3Q|$, we obtain $|f|_{3Q}\le c_n|f|_{R_Q}$. Further,
setting ${\mathcal S}_j=\{R_Q\in {\mathscr D}^{(j)}:Q\in {\mathcal F}\}$, and using that ${\mathcal F}$ is $\frac{1}{2}$-sparse,
we obtain that each family ${\mathcal S}_j$ is $\frac{1}{2\cdot 9^n}$-sparse. It follows from (\ref{glob}) that
$$|[b,T]f(x)|\le c_nC_T\sum_{j=1}^{3^n}\sum_{R\in {\mathcal S}_j}\big(|b(x)-b_{R}||f|_{R}+|(b-b_{R})f|_{R}\big)\chi_R(x),$$
and therefore, the proof is complete.
\end{proof}

\section{Proof of Theorem \ref{weakcomm} and Corollary \ref{imprdep}}
Fix a dyadic lattice ${\mathscr D}$. Let ${\mathcal S}\subset {\mathscr D}$ be a sparse family.
Define the $L\log L$ sparse operator by
$${\mathcal A}_{\mathcal S,L\log L}f(x)=\sum_{Q\in {\mathcal S}}\|f\|_{L\log L,Q}\chi_Q(x).$$
It follows from (\ref{bmoest}) that
\begin{equation}\label{estt*}
|{\mathcal T}^{\star}_{b,\mathcal S}f(x)|\le c_n\|b\|_{BMO}{\mathcal A}_{\mathcal S,L\log L}f(x).
\end{equation}

Let $\Phi(t)=t\log({\rm{e}}+t)$. Given a Young function $\f$, denote
$$C_{\f}=\int_{1}^{\infty}\frac{\f^{-1}(t)}{t^2\log({\rm{e}}+t)}dt.$$

By Theorem \ref{commpoint} combined with (\ref{estt*}), Lemma \ref{unsparse} and a submultiplicative property of $\Phi$ expressed in (\ref{subm}), Theorem \ref{weakcomm}
is an immediate consequence of the following two lemmas.

\begin{lemma}\label{lloglweak}
Suppose that ${\mathcal S}$ is $\frac{31}{32}$-sparse. Let $\f$ be a Young function such that $C_{\f}<\infty$.
Then for an arbitrary weight $w$,
$$
w_{{\mathcal A}_{\mathcal S,L\log L}f}(\la)\le cC_{\f}\int_{\mathbb{R}^{n}}\Phi\left(\frac{|f(x)|}{\lambda}\right)M_{(\Phi\circ\f)(L)}w(x)dx\quad(\la>0),
$$
where $c>0$ is an absolute constant.
\end{lemma}

\begin{lemma}\label{tbf} Let $b\in BMO$. Suppose that ${\mathcal S}$ is $\frac{7}{8}$-sparse.
Let $\f$ be a Young function such that $C_{\f}<\infty$. Then for an arbitrary weight $w$,
$$
w_{{\mathcal T}_{b,\mathcal S}f}(\la)\le \frac{c_nC_{\f}\|b\|_{BMO}}{\la}\int_{{\mathbb R}^n}|f(x)|M_{(\Phi\circ\f)(L)}w(x)dx\quad(\la>0).
$$
\end{lemma}

In the following subsection we separate a common ingredient used in the proofs of both Lemmas \ref{lloglweak} and \ref{tbf}.

\subsection{The key lemma}
Assume that $\Psi$ is a Young function satisfying
\begin{equation}\label{psicond}
\Psi(4t)\le {\Lambda}_{\Psi}\Psi(t)\quad(t>0, \Lambda_{\Psi}\ge 1).
\end{equation}

Given a dyadic lattice ${\mathscr D}$ and $k\in {\mathbb N}$, denote
$${\mathcal F}_k=\{Q\in {\mathscr D}:4^{k-1}<\|f\|_{\Psi,Q}\le 4^k\}.$$

The following lemma in the case $\Psi(t)=t$ was proved in \cite{DLR}. Our extension to any Young function satisfying (\ref{psicond}) is based on similar ideas.
Notice that the main cases of interest for us are $\Psi(t)=t$ and $\Psi(t)=\Phi(t)$.

\begin{lemma}\label{keylemma} Suppose that the family ${\mathcal F}_k$ is $\Big(1-\frac{1}{2\Lambda_{\Psi}}\Big)$-sparse.
Let $w$ be a weight and let $E$ be an arbitrary measurable set with $w(E)<\infty$. Then, for every Young function $\f$,
$$
\int_{E}\Big(\sum_{Q\in {\mathcal F}_k}\chi_Q\Big)wdx\le 2^{k}w(E)+\frac{4\Lambda_{\Psi}}{\bar\f^{-1}((2\Lambda_{\Psi})^{2^k})}\int_{{\mathbb R}^n}\Psi(4^k|f|)M_{\f(L)}wdx.
$$
\end{lemma}

\begin{proof} By Fatou's lemma, one can assume that the family ${\mathcal F}_k$ is finite.
Split $\mathcal{F}_{k}$ into the layers $\mathcal{F}_{k,\nu}$, $\nu=0,1,\dots$,
where $\mathcal{F}_{k,0}$ is the family of the maximal cubes in $\mathcal{F}_{k}$
and $\mathcal{F}_{k,\nu+1}$ is the family of the maximal cubes in $\mathcal{F}_{k}\setminus\bigcup_{l=0}^{\nu}\mathcal{F}_{k,l}$.

Denote $E_{Q}=Q\setminus\bigcup_{Q'\in\mathcal{F}_{k,\nu+1}}Q'$
for each $Q\in\mathcal{F}_{k,\nu}$. Then the sets $E_Q$ are pairwise disjoint for $Q\in {\mathcal F}_k$.

For $\nu\ge 0$ and $Q\in {\mathcal F}_{k,\nu}$ denote
$$A_k(Q)=\bigcup_{Q'\in\mathcal{F}_{k,\nu+2^k},Q'\subset Q}Q'.$$
Observe that
$$Q\setminus A_{k}(Q)=\bigcup_{l=0}^{2^k-1}\bigcup_{Q'\in\mathcal{F}_{k,\nu+l},Q'\subseteq Q}E_{Q'}.$$

Using the disjointness of the sets $E_Q$, we obtain
\begin{eqnarray}
\sum_{Q\in {\mathcal F}_k}
w\big(E\cap (Q\setminus A_{k}(Q))\big)&\le&
\sum_{\nu=0}^{\infty}\sum_{Q\in\mathcal{F}_{k,\nu}}\sum_{l=0}^{2^k-1}\sum_{\stackrel{{\scriptstyle Q'\in\mathcal{F}_{k,\nu+l}}}{Q'\subseteq Q}}w(E\cap E_{Q'})\nonumber\\
&\le& 2^k\sum_{Q\in {\mathcal F}_k}w(E\cap E_Q)\le 2^{k}w(E).\label{onepart}
\end{eqnarray}

Now, let us show that
\begin{equation}\label{estimeq}
1\le \frac{2\Lambda_{\Psi}}{|Q|}\int_{E_Q}\Psi(4^k|f(x)|)dx\quad(Q\in{\mathcal S}_k).
\end{equation}

Fix a cube $Q\in {\mathcal F}_{k,\nu}$. Since $4^{-k-1}<\|f\|_{\Psi,Q}$, by (\ref{fact})
and by (\ref{psicond}),
\begin{equation}\label{cpsi}
1<\frac{1}{|Q|}\int_Q\Psi(4^{k+1}|f|)\le \frac{\Lambda_{\Psi}}{|Q|}\int_Q\Psi(4^{k}|f|).
\end{equation}
On the other hand, for any $P\in {\mathcal F}_k$ we have $\|f\|_{\Psi,P}\le 4^{-k}$, and hence, by (\ref{fact}),
$$\frac{1}{|P|}\int_P\Psi(4^{k}|f|)\le 1.$$
Using also that,  by the sparseness condition, $|Q\setminus E_Q|\le \frac{1}{2\Lambda_\Psi}|Q|$,
we obtain
\begin{eqnarray*}
&&\frac{1}{|Q|}\int_Q\Psi(4^{k}|f|)=\frac{1}{|Q|}\int_{E_Q}\Psi(4^{k}|f|)+\frac{1}{|Q|}\sum_{Q'\in {\mathcal S}_{k,\nu+1}}\int_{Q'}\Psi(4^{k}|f|)\\
&&\le \frac{1}{|Q|}\int_{E_Q}\Psi(4^{k}|f|)+\frac{|Q\setminus E_Q|}{|Q|}\le \frac{1}{|Q|}\int_{E_Q}\Psi(4^{k}|f|)+\frac{1}{2\Lambda_{\Psi}},
\end{eqnarray*}
which, along with (\ref{cpsi}), proves (\ref{estimeq}).

Applying the sparseness assumption again, we obtain $|A_k(Q)|\le(1/2\Lambda_{\Psi})^{2^k}|Q|$. From this and from H\"older's inequality (\ref{prop2}),
\begin{eqnarray*}
\frac{w(A_k(Q))}{|Q|}&\le& 2\|\chi_{A_k(Q)}\|_{\bar\f,Q}\|w\|_{\f,Q}=\frac{2}{\bar\f^{-1}(|Q|/|A_k(Q)|)}\|w\|_{\f,Q}\\
&\le& \frac{2}{\bar\f^{-1}((2\Lambda_{\Psi})^{2^k})}\|w\|_{\f,Q}.
\end{eqnarray*}
Combining this with (\ref{estimeq}) yields
$$
w(A_k(Q))\le \frac{4\Lambda_{\Psi}}{\bar\f^{-1}((2\Lambda_{\Psi})^{2^k})}\int_{E_Q}\Psi(4^k|f|)M_{\f(L)}wdx.
$$
Hence, by the disjointness of the sets $E_Q$,
$$
\sum_{Q\in {\mathcal F}_k}
w(A_k(Q))\le \frac{4\Lambda_{\Psi}}{\bar\f^{-1}((2\Lambda_{\Psi})^{2^k})}\int_{{\mathbb R}^n}\Psi(4^k|f|)M_{\f(L)}wdx,
$$
which, along with (\ref{onepart}), completes the proof.
\end{proof}

\subsection{Proof of Lemmas \ref{lloglweak} and \ref{tbf}} We first mention another common ingredient used in both proofs.
\begin{prop}\label{common} Let $\Psi$ be a Young function. Assume that $G$ is an operator such that for every Young function $\f$,
\begin{equation}\label{coming}
w_{Gf}(\la)\le \Big(\int_1^{\infty}\frac{\f^{-1}(t)}{t^2}dt\Big)\int_{\mathbb{R}^{n}}\Psi\left(\frac{|f(x)|}{\lambda}\right)M_{\f(L)}w(x)dx.
\end{equation}
Then
$$w_{Gf}(\la)\le cC_{\f}\int_{\mathbb{R}^{n}}\Psi\left(\frac{|f(x)|}{\lambda}\right)M_{(\Phi\circ\f)(L)}w(x)dx,$$
where $c>0$ is an absolute constant, and $C_{\f}=\int_{1}^{\infty}\frac{\f^{-1}(t)}{t^2\log({\rm{e}}+t)}dt.$
\end{prop}

Indeed, this follows immediately by setting $\Phi\circ\f$ instead of $\f$ in (\ref{coming}) and observing that $(\Phi\circ\f)^{-1}=\f^{-1}\circ\Phi^{-1}$ and
\begin{equation}\label{cing}
\int_1^{\infty}\frac{\f^{-1}\circ\Phi^{-1}(t)}{t^2}dt=
\int_{\Phi^{-1}(1)}^{\infty}\frac{\f^{-1}(t)}{\Phi(t)^2}\Phi'(t)dt\le cC_{\f}.
\end{equation}

Turn to Lemma \ref{lloglweak}. We actually obtain a stronger statement, namely, we
will prove the following.

\begin{lemma}\label{lloglsparse}
Suppose that ${\mathcal S}$ is $\frac{31}{32}$-sparse. Let $\f$ be a Young function such that
$$K_{\f}=\int_1^{\infty}\frac{\f^{-1}(t)\log\log({\rm e}^2+t)}{t^2\log({\rm e}+t)}dt<\infty.$$
Then for an arbitrary weight $w$,
$$
w_{{\mathcal A}_{\mathcal S,L\log L}f}(\la)\le cK_{\f}\int_{\mathbb{R}^{n}}\Phi\left(\frac{|f(x)|}{\lambda}\right)M_{\f(L)}w(x)dx\quad(\la>0),
$$
where $c>0$ is an absolute constant.
\end{lemma}

Since $K_{\f}\le \int_1^{\infty}\frac{\f^{-1}(t)}{t^2}dt$, Proposition \ref{common} shows that
Lemma \ref{lloglweak} follows from Lemma \ref{lloglsparse}.

\begin{proof}[Proof of Lemma \ref{lloglsparse}]
Consider the set
$$E=\{x\in {\mathbb R}^n:{\mathcal A}_{\mathcal S,L\log L}f(x)>4, M_{L\log L}f(x)\le 1/4\}.$$
By homogeneity combined with Remark \ref{llog}, it suffices to prove that
\begin{equation}\label{suffic}
w(E)\le cK_{\f}\int_{\mathbb{R}^{n}}\Phi(|f(x)|)M_{\varphi(L)}w(x)dx.
\end{equation}
One can assume that $w(E)<\infty$
(otherwise, one could first obtain (\ref{suffic}) for $E\cap K$ instead of $E$, for any compact set $K$).

Denote
$${\mathcal S}_k=\{Q\in {\mathcal S}:4^{-k-1}<\|f\|_{L\log L,Q}\le 4^{-k}\}$$
and set
$$T_kf(x)=\sum_{Q\in {\mathcal S}_k}\|f\|_{L\log L,Q}\chi_Q(x).$$
If $E\cap Q\not=\emptyset$ for some $Q\in {\mathcal S}$, then $\|f\|_{L\log L,Q}\le 1/4$. Therefore, for $x\in E$,
\begin{equation}\label{repr}
{\mathcal A}_{\mathcal S,L\log L}f(x)=\sum_{k=1}^{\infty}T_kf(x).
\end{equation}

Now we apply Lemma \ref{keylemma} with $\Psi=\Phi$ and ${\mathcal F}_k={\mathcal S}_k$.
Notice that, by (\ref{subm}), one can take $\Lambda_{\Psi}=16$ in (\ref{psicond}) and
$\Phi(4^k|f|)\le ck4^k\Phi(|f|)$. Combining this with $T_kf(x)\le 4^{-k}\sum_{Q\in {\mathcal S}_k}\chi_Q$,
by Lemma \ref{keylemma} we obtain
$$
\int_{E}(T_{k}f)wdx\le 2^{-k}w(E)+\frac{ck}{\bar\f^{-1}(2^{2^{k}})}\int_{\mathbb{R}^{n}}\Phi(|f(x)|)M_{\varphi(L)}w(x)dx.
$$

Combining (\ref{repr}) with the latter estimate implies,
\begin{eqnarray*}
w(E)&\le& \frac{1}{4}\int_E({\mathcal A}_{\mathcal S,L\log L}f)wdx\le \frac{1}{4}\sum_{k=1}^{\infty}\int_{E}(T_{k}f)wdx\\
&\le&\frac{1}{4}w(E)+c\left(\sum_{k=1}^{\infty}\frac{k}{\bar\f^{-1}(2^{2^{k}})}\right)\int_{\mathbb{R}^{n}}\Phi(|f(x)|)M_{\varphi(L)}w(x)dx.
\end{eqnarray*}
From this,
$$
w(E)\le c\left(\sum_{k=1}^{\infty}\frac{k}{\bar\f^{-1}(2^{2^{k}})}\right)\int_{\mathbb{R}^{n}}\Phi(|f(x)|)M_{\varphi(L)}w(x)dx.
$$
Next, using that $\bar\f^{-1}(t)\f^{-1}(t)\approx t$, we obtain
$$
\sum_{k=1}^{\infty}\frac{k}{\bar\f^{-1}(2^{2^{k}})}\le c\sum_{k=1}^{\infty}\int_{2^{2^{k-1}}}^{2^{2^k}}\frac{\log\log({\rm e}^2+t)}{\bar\f^{-1}(t)t\log({\rm e}+t)}dt\le cK_{\f},
$$
which, along with the previous estimate, yields (\ref{suffic}), and therefore, the proof is complete.
\end{proof}

\begin{proof}[Proof of Lemma \ref{tbf}]
Denote
$$E=\{x:|{\mathcal T}_{b,\mathcal S}f(x)|>8,Mf(x)\le 1/4\}.$$
By the Fefferman-Stein estimate (\ref{fs}) and by homogeneity, it suffices to assume that $\|b\|_{BMO}=1$ and to show that in this case,
$$
w(E)\le cC_{\f}\int_{{\mathbb R}^n}|f|M_{(\Phi\circ\f)(L)}wdx.
$$

Let
$${\mathcal S}_k=\{Q\in {\mathcal S}:4^{-k-1}<|f|_Q\le 4^{-k}\}$$
and for $Q\in {\mathcal S}_k$, set
$$F_k(Q)=\{x\in Q:|b(x)-b_Q|>(3/2)^k\}.$$
If $E\cap Q\not=\emptyset$ for some $Q\in {\mathcal S}$, then $\|f\|_{Q}\le 1/4$. Therefore, for $x\in E$,
\begin{eqnarray*}
&&|{\mathcal T}_{b,\mathcal S}f(x)|\le\sum_{k=1}^{\infty}\sum_{Q\in {\mathcal S}_k}|b(x)-b_Q||f|_Q\chi_Q(x)\\
&&\le \sum_{k=1}^{\infty}(3/2)^k\sum_{Q\in {\mathcal S}_k}|f|_Q\chi_Q(x)+\sum_{k=1}^{\infty}\sum_{Q\in {\mathcal S}_k}|b(x)-b_Q||f|_Q\chi_{F_k(Q)}(x)\\
&&\equiv {\mathcal T}_1f(x)+{\mathcal T}_2f(x).
\end{eqnarray*}

Let $E_i=\{x\in E:{\mathcal T}_if(x)>4\}, i=1,2.$ Then
\begin{equation}\label{twot}
w(E)\le w(E_1)+w(E_2).
\end{equation}

Lemma \ref{keylemma} with $\Psi(t)=t$ yields (with any Young function $\f$)
$$
\int_{E_1}({\mathcal T}_1f)wdx\le \Big(\sum_{k=1}^{\infty}(3/4)^k\Big)w(E_1)+16\Big(\sum_{k=1}^{\infty}\frac{(3/2)^k}{\bar\f^{-1}(2^{2^k})}\Big)\int_{{\mathbb R}^n}|f|M_{\f(L)}wdx.
$$
This estimate, combined with $w(E_1)\le \frac{1}{4}\int_{E_1}({\mathcal T}_1f)wdx$, implies
$$
w(E_1)\le 16\Big(\sum_{k=1}^{\infty}\frac{(3/2)^k}{\bar\f^{-1}(2^{2^k})}\Big)\int_{{\mathbb R}^n}|f|M_{\f(L)}wdx.
$$
Since $\bar\f^{-1}(t)\f^{-1}(t)\approx t$, we obtain
\begin{eqnarray*}
\sum_{k=1}^{\infty}\frac{(3/2)^k}{\bar\f^{-1}(2^{2^k})}\le c\sum_{k=1}^{\infty}\int_{2^{2^{k-1}}}^{2^{2^k}}
\frac{1}{\bar\f^{-1}(t)}\frac{dt}{t}\le c\int_1^{\infty}\frac{\f^{-1}(t)}{t^2}dt.
\end{eqnarray*}
Hence,
$$
w(E_1)\le c\Big(\int_1^{\infty}\frac{\f^{-1}(t)}{t^2}dt\Big)\int_{{\mathbb R}^n}|f|M_{\f(L)}wdx,
$$
which by Proposition \ref{common} yields
\begin{equation}\label{oneparttt}
w(E_1)\le cC_{\f}\int_{{\mathbb R}^n}|f|M_{(\Phi\circ\f)(L)}wdx.
\end{equation}

Turn to the estimate of $w(E_2)$. Exactly as in the proof of Lemma~\ref{keylemma}, for $Q\in {\mathcal S}_k$ define  disjoint subsets $E_Q$.
Then, by (\ref{estimeq}),
$$|f|_Q\le \frac{8}{|Q|}\int_{E_Q}|f|dx.$$
Hence,
\begin{eqnarray}
w(E_2)&\le& \frac{1}{4}\|{\mathcal T}_2f\|_{L^1}\label{we2}\\
&\le& 2\sum_{k=1}^{\infty}\sum_{Q\in {\mathcal S}_k}\Big(\frac{1}{|Q|}\int_{F_k(Q)}|b-b_Q|wdx\Big)\int_{E_Q}|f|\nonumber.
\end{eqnarray}

Now we apply twice the generalized H\"older inequality. First, by (\ref{bmoest}),
\begin{equation}\label{firsth}
\frac{1}{|Q|}\int_{F_k(Q)}|b-b_Q|wdx\le c_n\|w\chi_{F_k(Q)}\|_{L\log L,Q}.
\end{equation}
Second, we use (\ref{prop3}) with $C(t)=\Phi(t), B(t)=\Phi\circ\f(t)$ and $A$ defined by
$$A^{-1}(t)=\frac{C^{-1}(t)}{B^{-1}(t)}=\frac{\Phi^{-1}(t)}{\f^{-1}\circ\Phi^{-1}(t)}.$$
Since $\f(t)/t$ and $\Phi$ are strictly increasing functions, $A$ is strictly increasing, too.
Hence, by (\ref{prop3}), we obtain
\begin{eqnarray}
\|w\chi_{F_k(Q)}\|_{L\log L,Q} &\le& 2\|\chi_{F_k(Q)}\|_{A,Q}\|w\|_{(\Phi\circ\f),Q}\label{secondh}\\
&=&\frac{2}{A^{-1}(|Q|/|F_k(Q)|)}\|w\|_{(\Phi\circ\f),Q}.\nonumber
\end{eqnarray}

By the John-Nirenberg inequality (\ref{jn}), $|F_k(Q)|\le \a_k|Q|,$
where $\a_k=\min(1,e^{-\frac{(3/2)^k}{2^ne}+1})$. Combining this with (\ref{firsth}) and (\ref{secondh}) yields
$$
\frac{1}{|Q|}\int_{F_k(Q)}|b-b_Q|wdx\le \frac{c_n}{A^{-1}(1/\a_k)}\|w\|_{(\Phi\circ\f),Q}.
$$
From this and from (\ref{we2}) we obtain
\begin{eqnarray*}
w(E_2)&\le& c_n\sum_{k=1}^{\infty}\frac{1}{A^{-1}(1/\a_k)}\sum_{Q\in {\mathcal S}_k}\|w\|_{(\Phi\circ\f),Q}\int_{E_Q}|f|\\
&\le& c_n\Big(\sum_{k=1}^{\infty}\frac{1}{A^{-1}(1/\a_k)}\Big)\int_{{\mathbb R}^n}|f|M_{(\Phi\circ\f)(L)}w(x)dx.
\end{eqnarray*}
Further, the sum on the right-hand side can be estimated as follows:
\begin{eqnarray*}
&&\sum_{k=1}^{\infty}\frac{1}{A^{-1}(1/\a_k)}\le c\sum_{k=1}^{\infty}\int_{1/\a_{k-1}}^{1/\a_k}\frac{1}{A^{-1}(t)}\frac{1}{t\log({\rm e}+t)}dt\\
&&\le c\int_1^{\infty}\frac{\f^{-1}\circ\Phi^{-1}(t)}{\Phi^{-1}(t)}\frac{1}{t\log({\rm e}+t)}dt\le c\int_1^{\infty}\frac{\f^{-1}\circ\Phi^{-1}(t)}{t^2}dt.
\end{eqnarray*}
Therefore, by (\ref{cing}),
$$
w(E_2)\le c_nC_{\f}\int_{{\mathbb R}^n}|f|M_{(\Phi\circ\f)(L)}w(x)dx,
$$
which, along with (\ref{twot}) and (\ref{oneparttt}), completes the proof.
\end{proof}

\subsection{Proof of Corollary \ref{imprdep}} The proof follows the same scheme as in the proof of \cite[Corollary 1.4]{HP}, and
hence we outline it briefly.

Using that $\log t\le \frac{t^{\a}}{\a}$ for $t\ge 1$ and $\a>0$, we obtain
$$M_{L(\log L)^{1+\e}}w(x)\le \frac{c}{\a^{1+\e}}M_{L^{1+(1+\e)\a}}w(x).$$
Next we use that for $r_n=1+\frac{1}{c_n[w]_{A_{\infty}}}$, $M_{L^{r_n}}w(x)\le 2Mw(x)$.
Hence, if $\a$ is such that $(1+\e)\a=\frac{1}{c_n[w]_{A_{\infty}}}$, then
$$\frac{1}{\e}M_{L(\log L)^{1+\e}}w(x)\le \frac{c_n}{\e}[w]_{A_{\infty}}^{1+\e}Mw(x)\le \frac{c_n}{\e}[w]_{A_{\infty}}^{1+\e}[w]_{A_1}w(x).$$
This estimate with $\e=1/\log({\rm{e}}+[w]_{A_{\infty}})$, along with (\ref{caseof}), completes the proof of Corollary \ref{imprdep}.

\section{Proof of Theorem \ref{bhlw}}
The main role in the proof is played by the following lemma. Denote by $\Omega(b;Q)$ the standard
mean oscillation,
$$\Omega(b;Q)=\frac{1}{|Q|}\int_Q|b-b_Q|dx.$$

\begin{lemma}\label{augm}
Let $\mathscr{D}$ be a dyadic lattice and let ${\mathcal{S}}\subset{\mathscr{D}}$
be a $\gamma$-sparse family. Assume that $b\in L_{loc}^{1}$. Then
there exists a $\frac{\gamma}{2\left(1+\gamma\right)}$-sparse family $\widetilde{\mathcal{S}}\subset{\mathscr{D}}$
such that ${\mathcal{S}}\subset\widetilde{\mathcal{S}}$ and for every
cube $Q\in\widetilde{\mathcal{S}}$,
\begin{equation}
|b(x)-b_{Q}|\le 2^{n+2} \sum_{R\in\widetilde{\mathcal{S}},R\subseteq Q}\Omega(b;R)\chi_{R}(x)\label{fuj}
\end{equation}
 for a.e. $x\in Q$.\end{lemma}

This lemma is based on several known ideas. The first idea is an estimate by oscillations over a sparse family
(see \cite{F1,Hyt2,L1}) and the second idea is an augmentation process (see Section 2.1).

\begin{proof}
Fix a cube $Q\in \mathscr{D}$. Let us show that there exists a (possibly empty) family of
pairwise disjoint cubes $\{P_j\}\in{\mathcal D}(Q)$ such that $\sum_j|P_j|\le \frac{1}{2}|Q|$ and for a.e. $x\in Q$,
\begin{equation}\label{itest}
|b(x)-b_{Q}|\le 2^{n+2}\O(b;Q)+\sum_j|b(x)-b_{P_j}|\chi_{P_j}.
\end{equation}

Consider the set
$$E=\Big\{x\in Q:M^d_{Q}(b-b_{Q})(x)>2^{n+2}\O(b;Q)\Big\},$$
where $M^d_{Q}$ is the standard dyadic local maximal operator restricted to a cube $Q$. Then
$|E|\le \frac{1}{2^{n+2}}|Q|$.

If $E=\emptyset$, then (\ref{itest}) holds trivially with the empty family $\{P_j\}$.
Suppose that $E\not=\emptyset$. The Calder\'on-Zygmund decomposition applied to the function $\chi_E$ on $Q$ at height $\la=\frac{1}{2^{n+1}}$
produces pairwise disjoint cubes $P_j\in {\mathcal D}(Q)$ such that
$$\frac{1}{2^{n+1}}|P_j|\le |P_j\cap E|\le \frac{1}{2}|P_j|$$
and $|E\setminus \cup_jP_j|=0$. It follows that $\sum_j|P_j|\le \frac{1}{2}|Q|$ and $P_j\cap E^{c}\not=\emptyset$.

Therefore,
\begin{equation}\label{estbyosc}
|b_{P_j}-b_{Q}|\le \frac{1}{|P_j|}\int_{P_j}|b-b_{Q}|dx\le 2^{n+2}\O(b;Q)
\end{equation}
and for a.e. $x\in Q$,
$$
|b(x)-b_{Q}|\chi_{Q\setminus\cup_jP_j}\le 2^{n+2}\O(b;Q).
$$
From this,
\begin{eqnarray*}
|b(x)-b_{Q}|\chi_{Q}&\le& |b(x)-b_{Q}|\chi_{Q\setminus\cup_jP_j}(x)+\sum_j|b_{P_j}-b_{Q}|\chi_{P_j}\\
&+&\sum_j|b(x)-b_{P_j}|\chi_{P_j}\\
&\le& 2^{n+2}\O(b;Q)+\sum_j|b(x)-b_{P_j}|\chi_{P_j},
\end{eqnarray*}
which proves (\ref{itest}).

We now observe that if $P_j\subset R$, where $R\in {\mathcal D}(Q)$, then $R\cap E^{c}\not=\emptyset$,
and hence $P_j$ in (\ref{estbyosc}) can be replaced by $R$, namely, we have
$$|b_{R}-b_{Q}|\le 2^{n+2}\O(b;Q).$$
Therefore, if $\cup_{j}P_j\subset \cup_iR_i$, where $R_i\in {\mathcal D}(Q)$, and the cubes $\{R_i\}$ are
pairwise disjoint, then exactly as above,
\begin{equation}\label{itest1}
|b(x)-b_{Q}|\le 2^{n+2}\O(b;Q)+\sum_i|b(x)-b_{R_i}|\chi_{R_i}.
\end{equation}

Iterating (\ref{itest}), we obtain that there exists a $\frac{1}{2}$-sparse
family ${\mathcal F}(Q)\subset {\mathcal D}(Q)$ such that for a.e. $x\in Q$,
$$
|b(x)-b_Q|\chi_Q\le 2^{n+2}\sum_{P\in {\mathcal F}(Q)}\O(b;P)\chi_P.
$$

We now augment ${\mathcal S}$ by families ${\mathcal F}(Q), Q\in {\mathcal S}$. Denote the resulting family by $\widetilde {\mathcal S}$. By Lemma \ref{Lem:Aug},
$\widetilde {\mathcal S}$ is $\frac{\gamma}{2\left(1+\gamma\right)}$-sparse.

Let us show that (\ref{fuj}) holds. Take an arbitrary cube $Q\in\widetilde {\mathcal S}$. Let $\{P_j\}$ be the cubes appearing in (\ref{itest}).
Denote by ${\mathcal M}(Q)$ the family of the maximal pairwise disjoint cubes from $\widetilde {\mathcal S}$
which are strictly contained in $Q$. Then, by the augmentation process, $\cup_jP_j\subset \cup_{P\in {\mathcal M(Q)}}P$. Therefore, by (\ref{itest1}),
\begin{equation}\label{newiter}
|b(x)-b_Q|\chi_Q\le 2^{n+2}\Omega(b;Q)+\sum_{P\in {{\mathcal M}(Q)}}|b(x)-b_P|\chi_P(x).
\end{equation}
Iterating this estimate completes the proof. Indeed, split $\widetilde {\mathcal S}(Q)=\{P\in \widetilde {\mathcal S}:P\subseteq Q\}$ into the layers $\widetilde {\mathcal S}(Q)=\cup_{k=0}^{\infty}{\mathcal M}_k$,
where ${\mathcal M}_0=Q$, ${\mathcal M}_1={\mathcal M}(Q)$ and ${\mathcal M}_k$ is the family of the maximal elements of ${\mathcal M}_{k-1}$.
Iterating (\ref{newiter}) $k$ times, we obtain
\begin{equation}\label{newit1}
|b(x)-b_Q|\chi_Q\le 2^{n+2}\sum_{P\in \widetilde {\mathcal S}(Q)}\Omega(b,P)\chi_P+\sum_{P\in {{\mathcal M}_k}}|b(x)-b_P|\chi_P(x).
\end{equation}
Now we observe that since $\widetilde {\mathcal S}$ is $\frac{\gamma}{2\left(1+\gamma\right)}$-sparse,
$$\sum_{P\in {\mathcal M}_k}|P|\le \frac{1}{(k+1)}\sum_{i=0}^k\sum_{P\in {\mathcal M}_i}|P|\le \frac{1}{(k+1)}\sum_{P\in \widetilde {\mathcal S}(Q)}|P|\le \frac{2(1+\gamma)}{\gamma(k+1)}|Q|.$$
Therefore, letting $k\to \infty$ in (\ref{newit1}), we obtain (\ref{fuj}).
\end{proof}

Recall the well-known (see \cite{CMP} or \cite{LN} for a different proof)  bound for the sparse operator ${\mathcal A}_{\mathcal S}$, where
${\mathcal S}$ is $\ga$-sparse:
\begin{equation}\label{asp}
\|{\mathcal A}_{{\mathcal S}}\|_{L^p(w)}\le c_{\gamma,n,p}[w]_{A_p}^{\max\big(1,\frac{1}{p-1}\big)}\quad(1<p<\infty).
\end{equation}

\begin{proof}[
Proof of Theorem \ref{bhlw}]
By Theorem \ref{commpoint} and by duality,
\begin{eqnarray}
&&\|[b,T]\|_{L^p(\mu)\to L^p(\la)}\label{dualest}\\
&&\le c_nC_T\sum_{j=1}^{3^n}\big(\|{\mathcal T}_{\mathcal S_j,b}\|_{L^p(\mu)\to L^p(\la)}+\|{\mathcal T}^{\star}_{\mathcal S_j,b}\|_{L^p(\mu)\to L^p(\la)}\big)\nonumber\\
&&=c_nC_T\sum_{j=1}^{3^n}\big(\|{\mathcal T}^{\star}_{\mathcal S_j,b}\|_{L^{p'}(\si_{\la})\to L^{p'}(\si_{\mu})}+\|{\mathcal T}^{\star}_{\mathcal S_j,b}\|_{L^p(\mu)\to L^p(\la)}\big),\nonumber
\end{eqnarray}
where ${\mathcal S}_j\subset {{\mathscr D}^{(j)}}$ is $\frac{1}{2\cdot 9^n}$-sparse.

By Lemma \ref{augm}, there are $\frac{1}{8\cdot 9^n}$-sparse families
$\widetilde{\mathcal S_j}$ containing ${\mathcal S}_j$, and also, for every cube $Q\in \widetilde{\mathcal S_j}$,
\begin{eqnarray*}
&&\int_Q|b(x)-b_Q||f|\le c_n\sum_{R\in \widetilde{\mathcal S}_j, R\subseteq Q}
\Omega(b;R)\int_R|f|\\
&&\le c_n\|b\|_{BMO_{\nu}}\sum_{R\in \widetilde{\mathcal S}_j, R\subseteq Q}|f|_R\nu(R)\le
c_n\|b\|_{BMO_{\nu}}\int_Q\big({\mathcal A}_{\widetilde{\mathcal S}_j}|f|\big)\nu dx.
\end{eqnarray*}
Therefore,
$$
{\mathcal T}_{\widetilde{\mathcal S}_j,b}^{\star}|f|(x)\le c_n\|b\|_{BMO_{\nu}}
{\mathcal A}_{{\widetilde{\mathcal S}_j}}\big(({\mathcal A}_{{\widetilde{\mathcal S}_j}}|f|)\nu\big)(x).
$$
Hence, applying (\ref{asp})  twice yields
\begin{eqnarray}
\|{\mathcal T}^{\star}_{\widetilde{\mathcal S}_j,b}\|_{L^p(\mu)\to L^p(\la)}&\le& c_{n,p}\|b\|_{BMO_{\nu}}\|{\mathcal A}_{\widetilde{\mathcal S}_j}\|_{L^p(\la)}\|{\mathcal A}_{\widetilde{\mathcal S}_j}\|_{L^p(\mu)}\label{case}\\
&\le& c_{n,p}\big([\la]_{A_p}[\mu]_{A_p}\big)^{\max\big(1,\frac{1}{p-1}\big)}\|b\|_{BMO_{\nu}}.\nonumber
\end{eqnarray}

From this and from the facts that $\nu=(\mu/\la)^{1/p}=(\si_{\la}/\si_{\mu})^{1/p'}$ and $[\si_w]_{A_{p'}}=[w]_{A_p}^{\frac{1}{p-1}}$, we obtain
\begin{eqnarray*}
\|{\mathcal T}^{\star}_{\widetilde{\mathcal S}_j,b}\|_{L^{p'}(\si_{\la})\to L^{p'}(\si_{\mu})}&\le& c_{n,p'}
\big([\si_{\mu}]_{A_{p'}}[\la_{\mu}]_{A_{p'}}\big)^{\max\big(1,\frac{1}{p'-1}\big)}\|b\|_{BMO_{\nu}}\\
&=& c_{n,p'}\big([\mu]_{A_p}[\la]_{A_p}\big)^{\max\big(1,\frac{1}{p-1}\big)}\|b\|_{BMO_{\nu}},
\end{eqnarray*}
It remains to combine this estimate with
 (\ref{dualest}) and (\ref{case}), and to observe that
${\mathcal T}_{{\mathcal S}_j,b}^{\star}|f(x)|\le
{\mathcal T}_{\widetilde{\mathcal S}_j,b}^{\star}|f(x)|$.
\end{proof}

\vskip 5mm
{\bf Acknowledgement.}
The third author thanks the Departamento de Matem\'atica of the Universidad Nacional del Sur, for the warm hospitality shown during his visit between February and May 2016.

\end{document}